\documentclass[11pt]{amsart}
\usepackage{amssymb}
\usepackage{amsmath}
\usepackage[active]{srcltx}
\usepackage{t1enc}
\usepackage[latin2]{inputenc}
\usepackage{verbatim}
\usepackage{amsmath,amsfonts,amssymb,amsthm}
\usepackage[mathcal]{eucal}
\usepackage{enumerate}
\usepackage[centertags]{amsmath}
\usepackage{graphics}

\newtheorem{theorem}{Theorem}

\newtheorem{lemma}{Lemma}
\newtheorem{remark}{Remark}

\newtheorem{corollary}{Corollary}

\newtheorem{proposition}{Proposition}

\begin{document}
\author{Davit Baramidze, Lars-Erik Persson and George Tephnadze}
\title[Fejér means]{Some now restricted maximal operators of Fejér means of Walsh-Fourier series  in the space $H_{1/2}$}

\address{D. Baramidze, The University of Georgia, School of science and technology, 77a Merab Kostava St, Tbilisi 0128, Georgia and Department of Computer Science and Computational Engineering, UiT - The Arctic University of Norway, P.O. Box 385, N-8505, Narvik, Norway.}
\email{davit.baramidze@ug.edu.ge }
\address{L.-E. Persson, UiT The Arctic University of Norway, P.O. Box 385, N-8505, Narvik, Norway and Department of Mathematics and Computer Science, Karlstad University, 65188 Karlstad, Sweden.}
\email{larserik6pers@gmail.com}
\address{G. Tephnadze, The University of Georgia, School of science and
	technology, 77a Merab Kostava St, Tbilisi 0128, Georgia.}
\email{g.tephnadze@ug.edu.ge}

\thanks{The research was supported by Shota Rustaveli National Science Foundation grant no. PHDF-21-1702.}
\date{}
\maketitle

\begin{abstract}
In this paper we derive the maximal subspace of natural numbers $\left\{ n_{k}:k\geq 0\right\} $, such that the restricted maximal operator, defined by ${\sup }_{k\in \mathbb{N}}\left\vert \sigma _{n_{k}}F\right\vert$ 
 on this subspace of Fejér means of Walsh-Fourier series  is bounded  from the martingale Hardy space $H_{1/2}$ to the Lebesgue space $L_{1/2}.$ The sharpness of this result is also proved.
\end{abstract}

\date{}



\section{INTRODUCTION}

\bigskip All symbols used in this introduction can be found in Section 2.

In the one-dimensional case, the weak (1,1)-type inequality for the
maximal operator $\sigma ^{\ast }$ of Fejér means $\sigma_n$ with respect to the Walsh system, defined by
\begin{equation*}
\sigma ^{\ast }f:=\sup_{n\in \mathbb{N}}\left\vert \sigma _{n}f\right\vert
\end{equation*}%
was investigated in Schipp \cite{Sc} and Pál, Simon \cite{PS} (see also \cite{BNPT}, \cite{NTT} and \cite{PTW2}). Fujii \cite{Fu} and Simon \cite{Si2} proved
that $\sigma ^{\ast }$ is bounded from $H_{1}$ to $L_{1}$. Weisz \cite{We2}
generalized this result and proved boundedness of $\sigma ^{\ast }$ from the
martingale space $H_{p}$ to the Lebesgue space $L_{p}$ for $p>1/2$. Simon 
\cite{Si1} gave a counterexample, which shows that boundedness does not hold
for $0<p<1/2.$ A counterexample for $p=1/2$ was given by Goginava \cite{GoAMH}. Moreover, in \cite{Goginava}  (see also \cite{tep1}) he  proved that there
exists a martingale $F\in H_{p}$ $\left( 0<p\leq 1/2\right) ,$ such that 
$$
\sup_{n\in \mathbb{N}}\left\Vert \sigma _{n}F\right\Vert _{p}=+\infty .
$$

Weisz \cite{We4} proved that
the maximal operator $\sigma ^{\ast }$  of the Fejér means is bounded from
the Hardy space $H_{1/2}$ to the space $weak-L_{1/2}$. 

The boundedness of
weighted maximal operators are considered in  
\cite{tep2}, \cite{tep3}. The results for summability of Fejér means of Walsh-Fourier series can be found in \cite{tb}, \cite{Mor}, \cite{na},   \cite{PTW1}, \cite{tep6}, \cite{We2}.

To study convergence of subsequences of  Fejér means and their restricted  maximal operators on the martingale Hardy spaces $H_p(G)$ for $0<p\leq 1/2,$ the central role is played by the fact that any natural number $n\in \mathbb{N}$ can be uniquely expression as
$$ 
n=\sum_{k=0}^{\infty }n_{j}2^{j},  \ \ n_{j}\in Z_{2} ~(j\in \mathbb{N}), $$ 
where only a finite numbers of $n_{j}$ differ from zero
and their important characters  $\left[ n\right],$ $\left\vert n\right\vert,$ $\rho\left( n\right)$  and $V(n)$ are defined by
\begin{equation*}
\left[ n\right] :=\min \{j\in \mathbb{N},n_{j}\neq 0\}, \ \ 
\left\vert n\right\vert :=\max \{j\in \mathbb{N},n_{j}\neq 0\}, \ \ 
\rho\left( n\right) =\left\vert n\right\vert -\left[ n\right] 
\end{equation*}
and
\begin{equation*}
V\left( n\right): =n_{0}+\overset{\infty }{\underset{k=1}{\sum }}\left|
n_{k}-n_{k-1}\right|, \text{ \ for \
	all \ \ }n\in \mathbb{N}.
\end{equation*}

Weisz \cite{We3} (see also \cite{We1}) also proved that for any $F\in H_p(G)$ $(p>0),$  the maximal operator 
$$
\underset{n\in \mathbb{N}}{\sup }\left\vert \sigma
_{2^n}F\right\vert
$$
is bounded from
the Hardy space $H_{p}$ to the Lebesgue space $L_{p}$, but (for details see \cite{PTWbook}) it is not bounded from the Hardy space  $H_{p}$ to the  space $H_{p}.$
Persson and Tephnadze \cite{PT} generalized this result and proved that if $0<p\leq 1/2$ and $\left\{ n_{k}:k\geq 0\right\} $ is a sequence of positive numbers, such that
	\begin{equation}\label{cond1}
	\sup_{k\in \mathbb{N}}\rho \left( n_{k}\right) \leq c<\infty,
	\end{equation}%
	then the maximal operator  $\widetilde{\sigma }^{\ast ,\nabla },$ defined by
		\begin{equation}\label{max}
	\widetilde{\sigma }^{\ast ,\nabla }F=\underset{k\in \mathbb{N}}{\sup }
	\left\vert \sigma _{n_{k}}F\right\vert,
	\end{equation}%
	is bounded from the Hardy space $H_{p}$ to the Lebesgue space $L_{p}.$
	Moreover, if $0<p<1/2$ and $\left\{ n_{k}:k\geq 0\right\} $ is a sequence of positive numbers, such that
\begin{equation*}
\sup_{k\in \mathbb{N}}\rho \left( n_{k}\right) =\infty, 
\end{equation*}
then there exists a martingale $F\in H_{p}$ such that 
$$
\underset{k\in \mathbb{N}}{\sup }\left\Vert \sigma _{n_{k}}F\right\Vert_{p}=\infty .
$$
From this fact it follows that if $0<p<1/2,$ $f\in H_{p}$ and $\left\{ n_{k}:k\geq 0\right\} $ is any sequence of positive numbers, then the maximal operator, defined by \eqref{max} is bounded from the Hardy space $H_{p}$ to the space $L_{p}$ if and only if the condition \eqref{cond1} is fulfilled. 
Moreover,  if $0<p<1/2,$ $f\in H_{p}$ and $\left\{ n_{k}:k\geq 0\right\} $ is any sequence of positive numbers, then $\sigma _{n_{k}}f$ are uniformly bounded	from the Hardy space $H_{p}$ to the Lebesgue space $L_{p}$ if and only if the condition \eqref{cond1} is fulfilled.  Moreover, condition \eqref{cond1} is necessary and sufficient condition for the boundedness of subsequence $\sigma _{n_{k}}f$ from the Hardy space $H_{p}$ to the Hardy space $H_{p}.$ 

It is easy to prove that condition  \eqref{cond1}  also implies that the maximal operator \eqref{max} is bounded from the Hardy space $H_{1/2}$ to the Lebesgue space $L_{1/2}.$ Moreover, for the subsequence $\{n_k=2^k+1\},$ we have that
$$
\rho \left( 2^k+1\right) =k\to\infty, \  \text{as} \   k\to\infty,
$$
but the maximal operator 
$$\underset{k\in \mathbb{N}}{\sup }
\left\vert \sigma _{2^k+1}F\right\vert
$$
is bounded from the Hardy space $H_{1/2}$ to the Lebesgue space $L_{1/2}.$

In \cite{tep8} it was proved that if $F\in H_{1/2},$ then there exists an absolute
	constant $c,$ such that%
	\begin{equation*}
	\left\Vert \sigma _{n}F\right\Vert _{H_{1/2}}\leq cV^{2}\left( n\right)
	\left\Vert F\right\Vert _{H_{1/2}}.
	\end{equation*}	
Moreover, if  $\left\{ n_{k}:k\geq 0\right\} $ is a subsequence of positive integers $\mathbb{N}_{+},$ such that 
$$\sup_{k\in \mathbb{N}}V\left(n_{k}\right) =\infty ,$$
and 
$\Phi_n$ is any nondecreasing, nonnegative sequence, satisfying the conditions $\Phi_n \uparrow \infty $ and
	\begin{equation*}
	\overline{\underset{k\rightarrow \infty }{\lim }}\frac{V^{2}\left(
		n_{k}\right) }{\Phi_{n_{k} }}=\infty, 
	\end{equation*}%
then there exists a martingale $F\in H_{1/2},$ such that
	\begin{equation*}
	\underset{k\in \mathbb{N}}{\sup }\left\Vert \frac{\sigma _{n_{k}}F}{\Phi_{n_{k}} }\right\Vert _{1/2}=\infty .
	\end{equation*}
It follows that if $f\in H_{1/2}$ and $\left\{ n_{k}:k\geq 0\right\} $ is any sequence of positive numbers, then $\sigma _{n_{k}}f$ are bounded	from the Hardy space $H_{1/2}$ to the space $H_{1/2}$ if and only if, for some $c,$ 
$$\sup_{k\in \mathbb{N}}V\left(n_{k}\right)<c<\infty.$$ 

Convergence of subsequences of partial sums and Fejér means of Walsh-Fourier
series can be found in \cite{PTT1}, \cite{tep4}, \cite{tep5} and \cite{tut1}.

In this paper we complement the reported research above by investigating the limit case $p=1/2.$ In particular, we derive the
 maximal subspace of natural numbers $\left\{ n_{k}:k\geq
0\right\} $, such that restricted maximal operator, defined by ${\sup }_{k\in \mathbb{N}}\left\vert \sigma _{n_{k}}F\right\vert$  on this subspace 
of Fejér means of Walsh-Fourier series  is bounded  from the martingale Hardy space $H_{1/2}$ to the Lebesgue space $L_{1/2}.$

This paper is organized as follows: In order not to disturb our discussions later on some
definitions and notations are presented in Section 2. The main result and some of its
consequences can be found in Section 3. For the proof of the main result we need some auxiliary statements, some of them are new and of independent interest (see Propositions 1 and 2). These results are
presented in Section 4. The detailed proofs are given in Section 5.

\section{Definitions and Notations}

Let $\mathbb{N}_{+}$ denote the set of the positive integers, $\mathbb{N}:=%
\mathbb{N}_{+}\cup \{0\}.$ Denote by $Z_{2}$ the discrete cyclic group of
order 2, that is $Z_{2}:=\{0,1\},$ where the group operation is the modulo 2
addition and every subset is open. The Haar measure on $Z_{2}$ is given so
that the measure of a singleton is 1/2.

Define the group $G$ as the complete direct product of the group $Z_{2},$
with the product of the discrete topologies of $Z_{2}$. The elements of $G$
are represented by sequences $x:=(x_{0},x_{1},...,x_{j},...),$ where $%
x_{k}=0\vee 1.$

It is easy to give a base for the neighborhood of $x\in G:$ 
\begin{equation*}
I_{0}\left( x\right) :=G,\text{ \ }I_{n}(x):=\{y\in
G:y_{0}=x_{0},...,y_{n-1}=x_{n-1}\}\text{ }(n\in \mathbb{N}).
\end{equation*}

Denote $I_{n}:=I_{n}\left( 0\right) ,$ $\overline{I_{n}}:=G$ $\backslash $ $%
I_{n}$ and $e_{n}:=\left( 0,...,0,x_{n}=1,0,...\right) \in G,$ for $n\in 
\mathbb{N}$. Then it is easy to prove that 
\begin{equation}
\overline{I_{M}}=\underset{i=0}{\bigcup\limits^{M-1}}I_{i}\backslash
I_{i+1}=\left( \overset{M-2}{\underset{k=0}{\bigcup }}\overset{M-1}{\underset%
{l=k+1}{\bigcup }}I_{l+1}\left( e_{k}+e_{l}\right) \right) \bigcup \left( 
\underset{k=0}{\bigcup\limits^{M-1}}I_{M}\left( e_{k}\right) \right) .
\label{1}
\end{equation}

If $n\in \mathbb{N},$ then every $n$ can be uniquely expressed as $%
n=\sum_{j=0}^{\infty }n_{j}2^{j},$ where $n_{j}\in Z_{2}$ $~(j\in \mathbb{N})
$ and only a finite numbers of $n_{j}$ differ from zero.

Let 
\begin{equation*}
\left[ n\right] :=\min \{j\in \mathbb{N},n_{j}\neq 0\}\text{ \ \ and \ \ \ }%
\left\vert n\right\vert :=\max \{j\in \mathbb{N},n_{j}\neq 0\},
\end{equation*}%
that is $2^{\left\vert n\right\vert }\leq n<2^{\left\vert n\right\vert +1}.$
Set 
\begin{equation*}
d\left( n\right) =\left\vert n\right\vert -\left[ n\right] ,\text{ \ for \
all \ \ }n\in \mathbb{N}.
\end{equation*}

Define the variation $V(n)$ of $n\in \mathbb{N},$ with binary coefficients $\left(n_{k},\ k\in \mathbb{N}\right) ,$ by 
\begin{equation*}
V\left( n\right): =n_{0}+\sum_{k=1}^{\infty }\left\vert
n_{k}-n_{k-1}\right\vert . 
\end{equation*}

Every $n\in \mathbb{N}$ can be also represented as $n=%
\sum_{i=1}^{r}2^{n_{i}},n_{1}>n_{2}>...n_{r}\geq 0.$ For such a representation of $n\in \mathbb{N},$ we denote numbers 
\begin{equation*}
n^{\left( i\right) }=2^{n_{i+1}}+...+2^{n_{r}},i=1,...,r.
\end{equation*}

Let 
$\begin{matrix}
2^{s}\le {{n}_{{{s}_{1}}}}\le {{n}_{{{s}_{2}}}}\le ...\le {{n}_{{{s}_{r}}}}\le {2^{s+1}}, & s\in \mathbb{N}.  \\
\end{matrix}$
For such $n_{s_j},$ which can be written as
$${{n}_{{{s}_{j}}}}=\sum\limits_{i=1}^{{{r}_{{{s}_{j}}}}}{\sum\limits_{k=l_{i}^{{{s}_{j}}}}^{t_{i}^{{{s}_{j}}}}{2^k}},$$ 
where   
$0\le l_{1}^{{{s}_{j}}}\le t_{1}^{{{s}_{j}}}\le l_{2}^{{{s}_{j}}}-2<l_{2}^{{{s}_{j}}}\le t_{2}^{{{s}_{j}}}\le ...\le l_{{{r}_{j}}}^{{{s}_{j}}}-2<l_{{{r}_{s_j}}}^{{{s}_{j}}}\le t_{{{r}_{s_j}}}^{{{s}_{j}}},$
we define
\begin{eqnarray}\label{As}
{{A}_{s}}&:=&\bigcup\limits_{j=1}^{r}{\left\{ l_{1}^{{{s}_{j}}},t_{1}^{{{s}_{j}}},l_{2}^{{{s}_{j}}},t_{2}^{{{s}_{j}}},...,l_{{{r}_{s_j}}}^{{{s}_{j}}}, t_{{{r}_{s_j}}}^{{{s}_{j}}} \right\}}\\ \notag
&=&\bigcup\limits_{j=1}^{r}{\left\{ l_{1}^{{{s}_{j}}},l_{2}^{{{s}_{j}}},...,l_{{{r}_{s_j}}}^{{{s}_{j}}} \right\}}
\bigcup\limits_{j=1}^{r}{\left\{ t_{1}^{{{s}_{j}}},t_{2}^{{{s}_{j}}},...,t_{{{r}_{s_j}}}^{{{s}_{j}}} \right\}}\\ \notag
&=&{\left\{ l_{1}^{s},l_{2}^{s},...,l_{r^1_{s}}^{s} \right\}}
\bigcup{\left\{ t_{1}^{s},t_{2}^{s},...,t_{{{r}_{s}^2}}^{s} \right\}}
\\ \notag
&=&{\left\{ u_{1}^{s},u_{2}^{s},...,u_{r^3_{s}}^{s} \right\}},
\end{eqnarray}
where $ u_{1}^{s}<u_{2}^{s}<...<u_{r^3_{s}}^{s}.$
We note that $
t_{{{r}_{s_j}}}^{{{s}_{j}}}=s\in {{A}_{s}}, \ \ \text{ for } \ \  j=1,2,...,r.  
$

Let us denote the cardinality of the set $A_s$ by $\vert A_s\vert$, that is
$$card(A_s):=\vert A_s\vert.$$ It is evident that
\begin{eqnarray*}
\vert A_s\vert=r^3_s\leq r_s^1+r_s^2.
\end{eqnarray*}
Moreover,
$r^s_3=card(A_s)<\infty$
if and only if
$r_s^1<\infty \ \  \text{and} \ \ r_s^2<\infty.$

We note that if $\vert A_s\vert<\infty,$ then each ${{n}_{{{s}_{j}}}}$ has bounded variation
\begin{equation}\label{cond1x}
V(n_{s_j})<c<\infty, \ \ \ \text{for each}  \ \ \ j=1,2,\ldots,r.
\end{equation}
and  \ \
$r_s^1<\infty, \  r_s^2<\infty \   \text{and} \  r_s^3<\infty.$

Therefore, if we consider blocks (intervals)
\begin{eqnarray}\label{000}
&& [u_1^s,u_1^s],  [u_1^s,u_2^s], ..., [u_1^s,u_{r_s^3}^s]\\ \notag
&& [u_2^s,u_2^s],  [u_2^s,u_3^s], ..., [u_2^s,u_{r_s^3}^s]\\ \notag
&&\ldots
\\  \notag
&&[u_{r_s^3}^s,u_{r_s^3}^s],
\end{eqnarray}
then it is easy to see that it contains  $\frac{{r_s^3}{(r_s^3+1)}}{2}$ different  blocks. Therefore, the dyadic representation of different natural numbers, which contains blocks from \eqref{000},
can be at most $2^{\frac{{r_s^3}{(r_s^3+1)}}{2}}-1,$ which is a finite number and the set
$\{{{n}_{{{s}_{1}}}},{{n}_{{{s}_{2}}}}, ...{{n}_{{{s}_{r}}}}\}$
is finite for all $s\in \mathbb{N}_+$, from which it follows that 
\begin{equation}\label{cond2x}
s_r<\infty, \ \ \ \text{for all}\ \ \ s\in\mathbb{N}.
\end{equation}

Summing up, we can conclude that
$\sup_{s\in\mathbb{N}}\vert A_s\vert<\infty$
if and only if the set  $\{{{n}_{{{s}_{1}}}},{{n}_{{{s}_{2}}}}, ...,{{n}_{{{s}_{r}}}}\}$ is finite for all $s\in \mathbb{N}_+$ and each ${{n}_{{{s}_{j}}}}$ has bounded variation, that is, conditions \eqref{cond1x} and \eqref{cond2x} are fullfiled.

The norms (or quasi-norm) of the spaces $L_{p}(G)$ and $L_{p,\infty }\left(
G\right),\ \left( 0<p<\infty \right) $ are, respectively, defined by 
\begin{equation*}
\left\Vert f\right\Vert _{p}^{p}:=\int_{G}\left\vert f\right\vert ^{p}d\mu
\ \ \text{and} \ \ \left\Vert f\right\Vert _{L_{p,\infty }(G)}^{p}:=\sup_{\lambda
>0}\lambda ^{p}\mu \left( f>\lambda \right) <+\infty ,
\end{equation*}

The $k$-th Rademacher function is defined by 
\begin{equation*}
r_{k}\left( x\right) :=\left( -1\right) ^{x_{k}}\text{\qquad }\left(
x\in G,\text{ }k\in \mathbb{N}\right) .
\end{equation*}

Now, define the Walsh system $w:=(w_{n}:n\in \mathbb{N})$ on $G$ as: 
\begin{equation*}
w_{n}(x):=\overset{\infty }{\underset{k=0}{\Pi }}r_{k}^{n_{k}}\left(
x\right) =r_{\left\vert n\right\vert }\left( x\right) \left( -1\right) ^{%
\underset{k=0}{\overset{\left\vert n\right\vert -1}{\sum }}n_{k}x_{k}}\text{%
\qquad }\left( n\in \mathbb{N}\right) .
\end{equation*}

The Walsh system is orthonormal and complete in $L_{2}\left( G\right) $ (see 
\cite{sws})$.$

If $f\in L_{1}\left( G\right) ,$ then we can define the Fourier coefficients,
partial sums of Fourier series, Fejér means, Dirichlet and Fejér kernels in
the usual manner: 
\begin{eqnarray*}
\widehat{f}\left( n\right) &:=&\int_{G}fw_{n}d\mu ,\,\,\left( n\in \mathbb{N}%
\right),\\
S_{n}f&:=&\sum_{k=0}^{n-1}\widehat{f}\left( k\right) w_{k},%
\text{\ }\left( n\in \mathbb{N}_{+},S_{0}f:=0\right) ,
\\
\sigma _{n}f&:=&\frac{1}{n}\sum_{k=1}^{n}S_{k}f,\text{ \ \ } \\
D_{n}&:=&\sum_{k=0}^{n-1}w_{k\text{ }},\text{ \ }\\
K_{n}&:=&\frac{1}{n}\overset{n}{%
\underset{k=1}{\sum }}D_{k}\text{ },\text{ \ }\left(n\in \mathbb{N}
_{+}\right) .
\end{eqnarray*}

Recall that (see \cite{sws})
\begin{equation}\label{1dn}
D_{2^{n}}\left( x\right) =\left\{ 
\begin{array}{ll}
2^{n} & \,\text{if\thinspace \thinspace \thinspace }x\in I_{n} \\ 
0\, & \ \,\text{if}\,\,x\notin I_{n}.%
\end{array}%
\right.  
\end{equation}

Let 
$n=\sum_{i=1}^{r}2^{n_{i}}, \ \ \ n_{1}>n_{2}>...>n_{r}\geq 0.$
Then (see \cite{G-E-S} and \cite{sws}) 
\begin{equation}
nK_{n}=\sum_{A=1}^{r}\left( \underset{j=1}{\overset{A-1}{\prod }}%
w_{2^{n_{j}}}\right) \left( 2^{n_{A}}K_{2^{n_{A}}}+n^{\left( A\right)
}D_{2^{n_{A}}}\right) .  \label{9a}
\end{equation}

The $\sigma $-algebra, generated by the intervals $\left\{ I_{n}\left(
x\right) :x\in G\right\} $ will be denoted by $\zeta _{n}$ $\left( n\in 
\mathbb{N}\right) .$ Denote by $F=\left( F_{n},n\in \mathbb{N}\right) $ a
martingale with respect to $\zeta _{n}$ $\left( n\in \mathbb{N}\right) $
(for details see e.g. \cite{We1}). The maximal function $F^{\ast }$ of a martingale $F$
is defined by 
\begin{equation*}
F^{\ast }:=\sup_{n\in \mathbb{N}}\left\vert F_{n}\right\vert .
\end{equation*}

In the case $f\in L_{1}\left( G\right),$ the maximal functions $f^{\ast }$ are also
given by 
\begin{equation*}
f^{\ast }\left( x\right) :=\sup\limits_{n\in \mathbb{N}}\left( \frac{1}{\mu
\left( I_{n}\left( x\right) \right) }\left\vert \int_{I_{n}\left( x\right)
}f\left( u\right) d\mu \left( u\right) \right\vert \right) .
\end{equation*}

For $0<p<\infty ,$ the Hardy martingale spaces $H_{p}$ $\left( G\right) $
consist of all martingales, for which 
\begin{equation*}
\left\Vert F\right\Vert _{H_{p}}:=\left\Vert F^{\ast }\right\Vert
_{p}<\infty .
\end{equation*}

A bounded measurable function $a$ is a $p$-atom, if there exists an interval $I,$
such that
\begin{equation*}
\text{ supp}\left( a\right) \subset I, \ \ \ \int_{I}ad\mu =0,\text{ \ }\left\Vert a\right\Vert _{\infty }\leq \mu \left(
I\right) ^{-1/p}.
\end{equation*}

It is easy to check that for every martingale $F=\left( F_{n},n\in \mathbb{N}%
\right) $ and every $k\in \mathbb{N}$ the limit
\begin{equation*}
\widehat{F}\left( k\right) :=\lim_{n\rightarrow \infty }\int_{G}F_{n}\left(
x\right) w_{k}\left( x\right) d\mu \left( x\right)
\end{equation*}%
exists and it is called the $k$-th Walsh-Fourier coefficients of $F.$

The Walsh-Fourier coefficients of $f\in L_{1}\left( G\right) $ are the same
as those of the martingale $\left( S_{2^{n}}f, n\in \mathbb{N}%
\right) $ obtained from $f$.

\section{The Main Result and its Consequences}

Our main result reads:
\begin{theorem}\label{theorem1}
a) Let $f\in {{H}_{1/2}}\left(G \right)$ and $\left\{ n_{k}:k\geq 0\right\} $ be a sequence of positive numbers and let 
$\left\{ n_{s_i}:\ 1\leq i\leq r\right\} \subset \left\{ n_{k}:\ k\geq 0\right\} $ be numbers such that
$
2^{s}\le {{n}_{{{s}_{1}}}}\le {{n}_{{{s}_{2}}}}\le ...\le {{n}_{{{s}_{r}}}}\le {2^{s+1}}, \  s\in \mathbb{N}.  $
If the sets ${{A}_{s}},$ defined by \eqref{As}, are finite for all $s\in \mathbb{N},$ that is the cardinality of the sets ${{A}_{s}}$ are finite:
 $${{\sup }_{s\in \mathbb{N}}}\vert{{A}_{s}}\vert<c<\infty ,$$ 
then the restricted maximal operator $\widetilde{\sigma }^{\ast ,\nabla },$ defined by
\begin{equation}\label{maxoperator}
\widetilde{\sigma }^{\ast ,\nabla }F=\underset{k\in \mathbb{N}}{\sup }
\left\vert \sigma _{n_{k}}F\right\vert,
\end{equation}
is bounded from the Hardy space ${{H}_{1/2}}$ to the Lebesgue space ${{L}_{1/2}}$.

b) (sharpness) Let  
\begin{equation}\label{cond2}
{{\sup }_{s\in \mathbb{N}}}\vert{{A}_{s}}\vert=\infty.
\end{equation}
Then there exists a martingale $f\in {{H}_{1/2}}\left(G \right),$ such that the maximal operator, defined by \eqref{maxoperator},  is not bounded from the Hardy space ${{H}_{1/2}}$ to the Lebesgue space ${{L}_{1/2}}.$
\end{theorem}

In particular, Theorem \ref{theorem1} implies the following optimal characterization:

\begin{corollary}\label{theorem3}
Let $F\in {{H}_{1/2}}\left( G \right)$ and $\left\{ n_{k}:k\geq 0\right\} $ be a sequence of positive numbers. Then the restricted maximal operator 
$\widetilde{\sigma }^{\ast ,\nabla },$
defined by \eqref{maxoperator}, is bounded  from the Hardy space ${{H}_{1/2}}$ to the Lebesgue space ${{L}_{1/2}}$ if and only if any sequence of positive numbers $\left\{ n_{k}:k\geq 0\right\} $ which satisfies $n_k\in [2^s, 2^{s+1}),$ is finite  for each $s\in\mathbb{N}_+$  and each  $\left\{ n_{k}:k\geq 0\right\} $  has bounded variation, i.e.
	\begin{eqnarray*}
		\sup_{k\in \mathbb{N}}V(n_{k})<c<\infty.
	\end{eqnarray*}
\end{corollary}

In order to be able to compare with some other results in the literature (see Remark \ref{theorem4}) we also state the following:
\begin{corollary}\label{theorem2}
Let $F\in {{H}_{1/2}}\left( G \right)$. Then the restricted maximal operators $\widetilde{\sigma }_i^{\ast ,\nabla }, \ i=1,2,3,$ defined by 
\begin{equation*}
\widetilde{\sigma }^{\ast ,\nabla }_1F=\underset{k\in \mathbb{N}}{\sup }
\left\vert \sigma _{2^k}F\right\vert,
\end{equation*}

\begin{equation}\label{maxoperator0}
\widetilde{\sigma }^{\ast ,\nabla }_2F=\underset{k\in \mathbb{N}}{\sup }
\left\vert \sigma _{2^k+1}F\right\vert,
\end{equation}

	\begin{equation}\label{maxoperator1}
	\widetilde{\sigma }^{\ast ,\nabla }_3F=\underset{k\in \mathbb{N}}{\sup }
	\left\vert \sigma _{2^k+2^{[k/2]}}F\right\vert,
	\end{equation}
where $[n]$ denotes the integer part of $n$,	are all bounded  from the Hardy space ${{H}_{1/2}}$ to the Lebesgue space ${{L}_{1/2}}$.
\end{corollary}
\begin{remark}\label{theorem4}
In \cite{PT} it was proved that if  $0<p<1/2,$ then the restricted maximal operators $\widetilde{\sigma }^{\ast ,\nabla }_2$ and $\widetilde{\sigma }^{\ast ,\nabla }_3,$ defined by \eqref{maxoperator0} and \eqref{maxoperator1}, are not bounded  from the Hardy space ${{H}_{p}}$ to the Lebesgue space $weak-{{L}_{p}}$.
	
On the other hand, in \cite{PT} it was proved that if $0<p\leq 1/2$, then the restricted maximal operator defined by  $\widetilde{\sigma }^{\ast ,\nabla }_1$ is bounded  from the Hardy space ${{H}_{p}}$ to the Lebesgue space $L_p.$ 
\end{remark}

We also have the following related consequence of Theorem \ref{theorem1}:
\begin{corollary}\label{theorem5}
Let $f\in {{H}_{1/2}}\left( G \right)$ and 
$\left\{ A^s_{k}:0\leq k\leq s-1\right\} $ be a sequence of positive numbers, defined by

$$\alpha^s_{0}=2^s+2^0, \ \ \ \alpha^s_{1}=2^s+2^0+2^1, ..., \alpha^s_{s-1}=2^s+2^0+2^1+...+2^{s-1}, \ \ \ s\in\mathbb{N}.$$
Then the restricted maximal operator $\widetilde{\sigma }^{\ast ,\nabla }_4,$ defined by 
\begin{equation*}
\widetilde{\sigma }^{\ast ,\nabla }_4F=\underset{s\in \mathbb{N}}{\sup }\ \underset{0\leq k\leq s-1}{\sup }
\left\vert \sigma _{\alpha^s_k}F\right\vert,
\end{equation*}
is not bounded  from the Hardy space ${{H}_{1/2}}$ to the Lebesgue space ${{L}_{1/2}}.$
\end{corollary}
\begin{remark}
We note that in this case  $V(\alpha_k^s)\leq 6$ for any $0\leq k\leq s-1$ and $s\in\mathbb{N},$ 

$${\left\{ l_{1}^{s},l_{2}^{s},...,l_{r^1_{s}}^{s} \right\}}={\left\{ 0,s \right\}}$$ 
i.e. it contains only two elements, the set 
$${\left\{ t_{1}^{s},t_{2}^{s},...,t_{{{r}_{s}^2}}^{s} \right\}}={\left\{ 0,1,2,3,...,s \right\}}$$
and
 
$${{\sup }_{s\in \mathbb{N}}}\vert{{A}_{s}}\vert={{\sup }_{s\in \mathbb{N}}}{(s+1)}=\infty .$$ 
\end{remark}

Finally, we state the following related result, which is connected to the research in \cite{tep8} (see Remark 4):
\begin{corollary}\label{theorem6}
	Let $f\in {{H}_{1/2}}\left( G \right)$ and 
	$\left\{ A^s_{k}:0\leq k\leq s-1\right\} $ be a sequence of positive numbers, defined by
	$$\beta^s_{s-1}=2^s+2^{s-1}, \ \ \ \beta^s_{s-2}=2^s+2^{s-1}+2^{s-2}, ..., \beta^s_{0}=2^s+2^{s-1}+...+2^{0}.$$
	Then the restricted maximal operator $\widetilde{\sigma }^{\ast ,\nabla }_5,$ defined by 
	\begin{equation*}
	\widetilde{\sigma }^{\ast ,\nabla }_5F=\underset{s\in \mathbb{N}}{\sup }\underset{0\leq k\leq s-1}{\sup }
	\left\vert \sigma _{\beta^s_k}F\right\vert,
	\end{equation*}
	is not bounded  from the Hardy space ${{H}_{1/2}}$ to the Lebesgue space ${{L}_{1/2}}.$
\end{corollary}
\begin{remark}
	We note that now $V(\beta_k^s)\leq 4,$ for any $0\leq k\leq s-1$ and $s\in\mathbb{N},$ 
	$${\left\{ t_{1}^{s},t_{2}^{s},...,t_{{{r}_{s}^2}}^{s} \right\}}={\left\{ 0,s \right\}}, \ \ \  \ \
{\left\{ l_{1}^{s},l_{2}^{s},...,l_{r^1_{s}}^{s} \right\}}={\left\{ 0,1,2,3,...,s \right\}}$$
	and 
	$${{\sup }_{s\in \mathbb{N}}}\vert{{A}_{s}}\vert={{\sup }_{s\in \mathbb{N}}}{(s+1)}=\infty .$$ 
\end{remark}
\begin{remark}\label{theorem50}
	Let $f\in {{H}_{1/2}}\left( G \right)$ and 
	$\left\{ \alpha^s_{k}:0\leq k\leq s-1\right\}$ and $\left\{ \beta^s_{k}:0\leq k\leq s-1\right\} $ be  a sequence of positive numbers, defined in Corollaries \ref{theorem5} and \ref{theorem6}.
	Then  there exist absolute constants $c_1$ and $c_2$ such that 
	\begin{equation*}
	\left\Vert \sigma _{\alpha^s_k}F\right\Vert_{H_{1/2}}\leq c_1 \left\Vert F\right\Vert_{H_{1/2}}
	\end{equation*}
and
\begin{equation*}
\left\Vert \sigma _{\beta^s_k}F\right\Vert_{H_{1/2}}\leq c_2 \left\Vert F\right\Vert_{H_{1/2}}
\end{equation*}
for any $ 0\leq k\leq s-1 \ \text{and}\  s\in\mathbb{N}.$

We note that these results were proved in \cite{tep8} and they follow, respectively, from the facts that 
$$V(\alpha_k^s)\leq 6 \ \ \ \text{and}\ \ \ V(\beta_k^s)\leq 4, \ \text{ for any} \ \ 0\leq k\leq s-1 \ \text{and}\  s\in\mathbb{N}.$$
\end{remark}

\section{Auxiliary Lemmas and Propositions}

\begin{lemma}[\textbf{Weisz \protect\cite{We3} (see also Simon \protect\cite%
{S})}]\label{lemma0}
A martingale $F=\left( F_{n},\text{ }n\in \mathbb{N}\right) $
is in $H_{p}\left( 0<p\leq 1\right) $ if and only if there exists a sequence 
$\left( a_{k},k\in \mathbb{N}\right) $ of p-atoms and a sequence $\left( \mu
_{k},k\in \mathbb{N}\right) $ of a real numbers, such that for every $n\in 
\mathbb{N},$
\begin{equation}
\qquad \sum_{k=0}^{\infty }\mu _{k}S_{2^{n}}a_{k}=F_{n},\text{ \ \ \ }%
\sum_{k=0}^{\infty }\left\vert \mu _{k}\right\vert ^{p}<\infty .  \label{6}
\end{equation}
Moreover, 

$$\left\Vert F\right\Vert _{H_{p}}\backsim \inf \left(
\sum_{k=0}^{\infty }\left\vert \mu _{k}\right\vert ^{p}\right) ^{1/p},$$
where the infimum is taken over all decomposition of $F$ of the form (\ref{6}).
\end{lemma}

\begin{lemma}[\textbf{Weisz \protect\cite{We1}}]
\label{lemma1}Suppose that an operator $T$ is $\sigma $-linear and
\end{lemma}
\begin{equation*}
\int\limits_{\overline{I}}\left\vert Ta\right\vert ^{p}d\mu \leq
c_{p}<\infty ,\text{ \ \ }\left( 0<p\leq 1\right)
\end{equation*}%
for every $p$-atom $a$, where  $I$ denote
the support of the atom. If $T$  is bounded from $L_{\infty \text{ 
}}$  to $L_{\infty },$  then 
\begin{equation*}
\left\Vert TF\right\Vert _{p}\leq c_{p}\left\Vert F\right\Vert _{H_{p}}.
\end{equation*}

\begin{lemma}[see e.g. \protect\cite{gat1}]\label{lemma2}
Let $t,n\in \mathbb{N}.$ Then 
\begin{equation*}
K_{2^{n}}\left( x\right) =\left\{ 
\begin{array}{c}
\text{ }2^{t-1},\text{\ if \ \ }x\in I_{n}\left( e_{t}\right) ,\text{ }n>t,%
\text{\ }x\in I_{t}\backslash I_{t+1}, \\ 
\left( 2^{n}+1\right) /2,\text{\ if \ \ }x\in I_{n}, \\ 
0,\text{\ otherwise.\ }%
\end{array}%
\right.
\end{equation*}
\end{lemma}

\begin{lemma}[see e.g. \cite{GoSzeged}, \cite{tep3}]
	\label{lemma4}Let  $n\geq 2^{M}$ and
	$$x\in I_{M}^{k,l}, \ \ k=0,...,M-1, \ \ l=k+1,...,M.$$
	Then
	\begin{equation*}
	\int_{I_{M}}\left\vert K_{n}\left( x+t\right) \right\vert d\mu \left(
	t\right) \leq cn2^{k+l-M}.
	\end{equation*}
\end{lemma}

\begin{lemma}[see \protect\cite{tep7}]\label{lemma5a}
Let 
$$n=\sum_{i=1}^{s}\sum_{k=l_{i}}^{t_{i}}2^{k}, \ \text{where } \ t_{1}\geq l_{1}>l_{1}-2\geq t_{2}\geq l_{2}>l_{2}-2>...>t_{s}\geq l_{s}\geq0.$$ 
	Then, for any $i=1,2,...,s,$
	\begin{equation*}
	n\left\vert K_{n}\left( x\right) \right\vert \geq 2^{2l_{i}-4},\text{ \ 
		for \  }x\in E_{l_{i}}:=I_{l_{i}+1}\left( e_{l_{i}-1}+e_{l_{i}}\right) , 
	\end{equation*}%
	 where 
	$I_{1}\left( e_{-1}+e_{0}\right) =I_{2}\left( e_{0}+e_{1}\right) .$
\end{lemma}

We also need the following new statement of independent interest:
\begin{proposition}
	\label{lemma3}Let 
	$$n=\sum_{i=1}^{s}\sum_{k=l_{i}}^{t_{i}}2^{k}, \ \text{
	where } \ 
	t_{1}\geq l_{1}>l_{1}-2\geq t_{2}\geq l_{2}>l_{2}-2>...>t_{s}\geq l_{s}\geq
	0.$$ 
	Then, for any $i=1,2,...,s,$
	\begin{equation*}
	n\left\vert K_{n}\left( x\right) \right\vert \geq 2^{2t_{i}-2},\text{ \ \
		for \ \ }x\in E_{t_{i}}:=I_{t_{i}+1}\left( e_{t_{i}+1}+e_{t_{i}+2}\right).
	\end{equation*}
\end{proposition}
\begin{proof} It is evident that we always have that $t_{i}+2\leq l_{i+1}.$
	If $t_{i}+2=l_{i+1},$ then 
	$E_{t_{i}}=I_{t_{i}+3}\left( e_{t_{i}+1}+e_{t_{i}+2}\right)=I_{l_{i+1}+1}\left( e_{l_{i+1}-1}+e_{l_{i+1}}\right)=E_{l_{i+1}}$
and if we apply Lemma \ref{lemma5a} we find that
\begin{equation*}
n\left\vert K_{n}\left( x\right) \right\vert \geq 2^{2l_{i+1}-4}=2^{2t_i},\text{ \ 
	for \  }x\in E_{l_{i+1}}=E_{t_{i}}.
\end{equation*}

Let $t_{i}+2<l_{i+1}.$ By combining \eqref{1dn} and Lemma \ref{lemma2}, for any $n\geq t_{i}+3$ we get that
$$D_{2^n}(x)=K_{2^n}(x)=0, \ \ \text{for} \ \ x\in E_{t_{i}}. $$

From (\ref{9a}), for $ x\in E_{t_{i}}$ we can conclude that 
\begin{eqnarray}\label{kn}
nK_{n} 
\end{eqnarray}
\begin{eqnarray*}
	&=&\sum_{r=1}^{s}\sum_{k=l_{r}}^{t_{r}}\left( \underset{j=i+1}{%
		\overset{s}{\prod }}\underset{q=l_{j}}{\overset{t_{j}}{\prod }}w_{2^{q}}%
	\underset{j=k+1}{\overset{t_{i}}{\prod }}w_{2^{j}}\right)\left( 2^{k}K_{2^{k}}+\left(
	\sum_{j=1}^{i-1}\sum_{q=l_{j}}^{t_{j}}2^{q}+\sum_{q=l_{i}}^{k-1}2^{q}\right)
	D_{2^{k}} \right)\\
	&=&\sum_{r=1}^{i}\sum_{k=l_{r}}^{t_{r}}\left( \underset{j=i+1}{%
		\overset{s}{\prod }}\underset{q=l_{j}}{\overset{t_{j}}{\prod }}w_{2^{q}}%
	\underset{j=k+1}{\overset{t_{i}}{\prod }}w_{2^{j}}\right)\left( 2^{k}K_{2^{k}}+\left(
	\sum_{j=1}^{i-1}\sum_{q=l_{j}}^{t_{j}}2^{q}+\sum_{q=l_{i}}^{k-1}2^{q}\right)
	D_{2^{k}} \right).
\end{eqnarray*}
Suppose that $l_i<t_i.$ Since
$$
\sum_{j=1}^{i-1}\sum_{q=l_{j}}^{t_{j}}2^{q}+\sum_{q=l_{i}}^{t_i-1}2^{q}\geq 2^{t_i-1}$$ 
for $x\in E_{t_{i}}$ we find that
\begin{eqnarray}\label{star0}
n\left\vert K_{n}\right\vert &\geq& \left\vert
2^{t_{i}}K_{2^{t_{i}}}+2^{t_i-1}
D_{2^{t_i}}\right\vert-\sum_{k=1}^{t_i-1}\left\vert 2^{k}K_{2^{k}}\right\vert
-\sum_{k=1}^{t_i-1}\left\vert 2^{k}D_{2^{k}}\right\vert \\ \notag
&:=&I-II-III.
\end{eqnarray}
Moreover, by combining \eqref{1dn} and Lemma \ref{lemma2} we get that
\begin{eqnarray} \label{10.023}
\ \ \ \ \  I\geq 2^{t_{i}}K_{2^{t_{i}}}\left( x\right)+2^{t_i-1}D_{2^{t_i}}=\frac{
	2^{2t_{i}}}{2}+2^{t_i-1}+\frac{2^{t_i}}{2}= 2^{2t_{i}}+2^{t_i-1}. 
\end{eqnarray}
For $II$ we have that
\begin{eqnarray}\label{10.1}
II\leq\sum_{k=0}^{t_{i}-1}2^{k}\frac{\left( 2^{k}+1\right) }{2}\leq \frac{1}{2}\frac{2^{2t_{i}}-1}{4-1}+\frac{1}{2}\frac{2^{t_{i}}-1}{2-1}
\leq \frac{2^{2t_{i}}}{6}+2^{t_{i}-1}.  
\end{eqnarray}
Moreover,  $III$ can be estimated as follows
\begin{equation}
III\leq \sum_{k=0}^{l_{i}-1}4^{k}=\frac{2^{2t_{i}}}{3}.
\label{10.2}
\end{equation}

By combining (\ref{10.023})-(\ref{10.2}) and putting them into \eqref{star0} we obtain that
\begin{equation}
n\left\vert K_{n}\left( x\right) \right\vert \geq I-II-III\geq \frac{%
	2^{2l_{i}}}{2}.  \label{10.3}
\end{equation}

If ${t_{i}}=l_i$ we get that $t_{i-1}\leq l_i-2={t_{i}}-2$. Hence, by using \eqref{kn} we find that
\begin{eqnarray}\label{star}
\ \ \ \ \ \ \ \ 	n\left\vert K_{n}\right\vert \geq \left\vert
	2^{t_{i}}K_{2^{t_{i}}}\right\vert-\sum_{k=1}^{t_i-2}\left\vert 2^{k}K_{2^{k}}\right\vert
	-\sum_{k=1}^{t_i-2}\left\vert 2^{k}D_{2^{k}}\right\vert
	:=I-II-III.
\end{eqnarray}
By simple calculations we get that
\begin{eqnarray}\label{10.000}
I\geq 2^{t_{i}}K_{2^{t_{i}}}\left( x\right)&=&\frac{%
	2^{2t_{i}}}{2}+\frac{2^{t_i}}{2},  
\end{eqnarray}
\begin{eqnarray}\label{10.10} 
\ \ \ \ \ II\leq \sum_{k=0}^{t_{i}-2}2^{k}\frac{\left( 2^{k}+1\right) }{2}
\leq\frac{1}{2}\frac{2^{2t_{i}-2}-1}{4-1}+\frac{1}{2}\frac{2^{t_{i}-1}-1}{2-1}
\leq \frac{2^{2t_{i}}}{24}+2^{t_{i}-2}  
\end{eqnarray}
and
\begin{equation}
III \leq \sum_{k=0}^{l_{i}-2}4^{k}=\frac{2^{2t_{i}}}{12}.
\label{10.20}
\end{equation}
We insert (\ref{10.000})-(\ref{10.20}) into \eqref{star} and find that
\begin{equation}
n\left\vert K_{n}\left( x\right) \right\vert \geq I-II-III\geq \frac{%
	2^{2l_{i}}}{4}.  \label{10.30}
\end{equation}

The proof is complete by just combining the estimates \eqref{10.3} and \eqref{10.30}.
\end{proof}

\begin{corollary}\label{corollary1}
Let 
	$$n=\sum_{i=1}^{s}\sum_{k=l_{i}}^{t_{i}}2^{k}, \ \text{where } \ t_{1}\geq l_{1}>l_{1}-2\geq t_{2}\geq l_{2}>l_{2}-2>...>t_{s}\geq l_{s}\geq0.$$ 
	Then, for any $i=1,2,...,s,$
\begin{equation*}
n\left\vert K_{n}\left( x\right) \right\vert \geq 2^{2t_{i}-5},\text{ \ \
	for \ \ }x\in E_{t_{i}}:=I_{t_{i}+1}\left( e_{t_{i}+1}+e_{t_{i}+2}\right).
\end{equation*}
and
	\begin{equation*}
n\left\vert K_{n}\left( x\right) \right\vert \geq 2^{2l_{i}-5},\text{ \ 
	for \  }x\in E_{l_{i}}:=I_{l_{i}+1}\left( e_{l_{i}-1}+e_{l_{i}}\right) , 
\end{equation*}
	where 
	$I_{1}\left( e_{-1}+e_{0}\right) =I_{2}\left( e_{0}+e_{1}\right) .$
\end{corollary}

Our second auxiliary result of independent interest is the following:
\begin{proposition}
\label{lemma5}Let 
$$n=\sum_{i=1}^{s}\sum_{k=l_{i}}^{t_{i}}2^{k}, \ \text{ where } \
t_{1}\geq l_{1}>l_{1}-2\geq t_{2}\geq l_{2}>l_{2}-2>...>t_{s}\geq l_{s}\geq
0.$$ 
Then 
\begin{eqnarray*}
\left\vert nK_{n}\right\vert &\leq& c\sum_{A=1}^{s}\left( 2^{l_{A}}
K_{2^{l_{A}}} +2^{t_{A}} K_{2^{t_{A}}}
+2^{l_A}\sum_{k=l_{A}}^{t_{A}}D_{2^{k}}\right).
%&\leq& c\sum_{A=1}^{s}\left( 2^{l_{A}}
%K_{2^{l_{A}}} +2^{t_{A}} K_{2^{t_{A}}}
%+2^{t_A}D_{2^{l_A}}\right)
\end{eqnarray*}
\end{proposition}
\begin{proof}
Let 
$n=\sum_{i=1}^{s}2^{n_{i}}, \ \ \ n_{1}>n_{2}>...>n_{r}\geq 0.$
By using (\ref{9a}) we get that
\begin{eqnarray}\label{star1}
nK_{n}&=&\sum_{A=1}^{s}\left( \underset{j=1}{\overset{A-1}{\prod }}%
w_{2^{n_{j}}}\right) \left( \left( 2^{n_{A}}K_{2^{n_{A}}}+\left(
2^{n_{A}}-1\right) D_{2^{n_{A}}}\right) \right)\\ \notag
&-&\sum_{A=1}^{s}\left( \underset{j=1}{\overset{A-1}{\prod }}%
w_{2^{n_{j}}}\right) \left( (2^{n_{A}}-1)-n^{\left( A\right) }\right)
D_{2^{n_{A}}}\\ \notag
&=&I_{1}-I_{2}.
\end{eqnarray}

For $I_{1}$ we have that

\begin{equation*}
I_{1}=\sum_{v=1}^{s}\left( \underset{j=v+1}{\overset{s}{\prod }}\underset{%
i=l_{j}}{\overset{t_{j}}{\prod }}w_{2^{i}}\right) \left(
\sum_{k=l_{v}}^{t_{v}}\left( \underset{j=k+1}{\overset{t_{v}}{\prod }}%
w_{2^{j}}\right) \left( 2^{k}K_{2^{k}}+\left( 2^{k}-1\right)
D_{2^{k}}\right) \right)
\end{equation*}%
\begin{equation*}
=\sum_{v=1}^{s}\left( \underset{j=v+1}{\overset{s}{\prod }}\underset{i=l_{j}}%
{\overset{t_{j}}{\prod }}w_{2^{i}}\right) \left(
\sum_{k=0}^{t_{v}}-\sum_{k=0}^{l_{v}-1}\right) \left( \underset{j=k+1}{%
\overset{t_{v}}{\prod }}w_{2^{j}}\right) \left( 2^{k}K_{2^{k}}+\left(
2^{k}-1\right) D_{2^{k}}\right)
\end{equation*}%
\begin{equation*}
=\sum_{v=1}^{s}\left( \underset{j=v+1}{\overset{s}{\prod }}\underset{i=l_{j}}%
{\overset{t_{j}}{\prod }}w_{2^{i}}\right) \left( \sum_{k=0}^{t_{v}}\left( 
\underset{j=k+1}{\overset{t_{v}}{\prod }}w_{2^{j}}\right) \left(
2^{k}K_{2^{k}}+\left( 2^{k}-1\right) D_{2^{k}}\right) \right)
\end{equation*}%
\begin{equation*}
-\sum_{v=1}^{s}\left( \underset{j=v+1}{\overset{s}{\prod }}\underset{i=l_{j}}{%
\overset{t_{j}}{\prod }}w_{2^{i}}\right) \left( \sum_{k=0}^{l_{v}-1}\left( 
\underset{j=k+1}{\overset{l_{v}-1}{\prod }}w_{2^{j}}\right) \left(
2^{k}K_{2^{k}}+\left( 2^{k}-1\right) D_{2^{k}}\right) \right) .
\end{equation*}

Since 
$2^{n}-1=\sum_{k=0}^{n-1}2^{k}$
and
\begin{equation*}
\left( 2^{n}-1\right) K_{2^{n}-1}=\sum_{k=0}^{n-1}\left(
\prod_{j=k+1}^{n-1}w_{2^{j}}\right) \left( 2^{k}K_{2^{k}}+\left(
2^{k}-1\right) D_{2^{k}}\right) ,
\end{equation*}%
\ we obtain that 
\begin{equation*}
I_{1}=\sum_{v=1}^{s}\left( \underset{j=v+1}{\overset{s}{\prod }}\underset{%
i=l_{j}}{\overset{t_{j}}{\prod }}w_{2^{i}}\right) \left(
2^{t_{v}+1}-1\right) K_{2^{t_{v}+1}-1}
\end{equation*}%
\begin{equation*}
-\sum_{v=1}^{s}\left( \underset{j=v+1}{\overset{s}{\prod }}\underset{i=l_{j}}{%
\overset{t_{j}}{\prod }}w_{2^{i}}\right) \left( 2^{l_{v}}-1\right)
K_{2^{l_{v}}-1}.
\end{equation*}%
Thus, by using the estimates
$$\left\vert K_{2^{n}}\right\vert \leq c\left\vert
K_{2^{n-1}}\right\vert   \ \ \ \ \ \text{ 
and} \ \ \ \ \  \left\vert K_{2^{n}-1}\right\vert \leq c\left\vert
K_{2^{n}}\right\vert, $$
we can conclude that
\begin{equation}
\left\vert I_{1}\right\vert \leq c\sum_{v=1}^{s}\left( 2^{l_{v}}\left\vert
K_{2^{l_{v}}}\right\vert +2^{t_{v}}\left\vert K_{2^{t_{v}}}\right\vert
\right) .  \label{12c}
\end{equation}

Let $l_{j}<n_{A}\leq t_{j},$ for some $j=1,...,s.$ Then
$$n^{\left( A\right)
}\geq \sum_{v=l_{j}}^{n_{A}-1}2^{v}\geq 2^{n_{A}}-2^{l_{j}} \ \ \ \ \ \ \ \text{
and} \ \ \ \ \ \ 2^{n_{A}}-1-n^{\left( A\right) }\leq 2^{l_{j}}.$$ 
If $l_{j}=n_{A}$ for some $%
j=1,...,s,$ then 
$$n^{\left( A\right) }\leq 2^{t_{j-1}+1}<2^{l_{j}}.$$
Hence, 
\begin{equation}
\left\vert I_{2}\right\vert \leq
c\sum_{v=1}^{s}2^{l_{v}}\sum_{k=l_{v}}^{t_{v}}D_{2^{k}}.  \label{12d}
\end{equation}
Finally, we use the estimates (\ref{12c})-(\ref{12d}) in \eqref{star1} and  the proof is complete.
\end{proof}

\section{Proof of the Theorem \ref{theorem1}}

\begin{proof}
Since $\sigma_{n}$ is bounded from $L_{\infty }$ to $%
L_{\infty },$
by Lemma \ref{lemma1}, the proof of theorem \ref{theorem1}
will be complete, if we prove that
\begin{equation}\label{star4}
\int_{\overline{I_{M}}}\left(\sup_{s_k\in \mathbb{N}}\left\vert \sigma _{n_{s_k}}a\left( x\right)
\right\vert \right)^{1/2}d\mu \left( x\right) \leq
c<\infty ,
\end{equation}
for every 1/2-atom $a.$ We may assume that $a$ is an arbitrary $1/2$-atom,
with support$\ I,$ $\mu \left( I\right) =2^{-M}$ and $I=I_{M}.$ It is easy
to see that 
$$\sigma _{n}\left( a\right) =0, \ \ \ \text{ when } \ \ \ n< 2^{M}.$$ 
Therefore, we can suppose that $n_{s_k}\geq 2^{M}.$

Let $x\in I_{M}$ and $2^{s}\leq n_{s_k}< 2^{s+1}$ for some $n_{s_k}\geq 2^M.$ Since $$\left\Vert a\right\Vert _{\infty }\leq 2^{2M},$$
by using Proposition \ref{lemma5} we obtain that 

\begin{eqnarray}\label{01}
\left\vert \sigma _{n_{s_k}}a\left( x\right) \right\vert &\leq& \int_{I_{M}}\left\vert a\left(t\right) \right\vert \left\vert K_{n_{s_k}}\left( x+t\right) \right\vert d\mu
\left( t\right) \\ \notag
&\leq& \left\Vert a\right\Vert _{\infty }
\int_{I_{M}}\left\vert K_{n_{s_k}}\left( x+t\right) \right\vert d\mu \left(
t\right) \\ \notag
&\leq& 2^{2M}\int_{I_{M}}\left\vert
K_{n_{s_k}}\left( x+t\right) \right\vert d\mu \left( t\right) 
\\ \notag
&\leq& \frac{2^{2M}}{n_{s_k}}\sum_{A=1}^{r_{s_k}}
\int_{I_{M}}2^{l_{A}^{s_k}}\left\vert K_{2^{l_A^{s_k}}}\left( x+t\right)
\right\vert d\mu \left( t\right)
\\ \notag
&+& \frac{2^{2M}}{n_{s_k}}\sum_{A=1}^{r_{s_k}}\int_{I_{M}}2^{t_{A}^{s_k}}\left\vert
K_{2^{t_A^{s_k}}}\left( x+t\right) \right\vert d\mu \left( t\right) 
\\ \notag
&+&\frac{2^{2M}}{n_{s_k}}\sum_{A=1}^{r_{s_k}}\int_{I_{M}}2^{l_{A}^{s_k}}
\sum_{k=l_{A}^{s_k}}^{\infty}D_{2^{k}}\left( x+t\right) d\mu \left( t\right)
\\ \notag
&\leq& \frac{2^M}{2^s}\left(2^M
\sum_{A=1}^{r_s^1}\int_{I_{M}}2^{l_{A}^s}\left\vert K_{2^{l^s_{A}}}\left( x+t\right)
\right\vert d\mu \left( t\right) \right)
\\ \notag
&+& \frac{2^M}{2^s}\left(2^M\sum_{A=1}^{r^2_{s}}\int_{I_{M}}2^{t^s_{A}}\left\vert
K_{2^{t^s_{A}}}\left( x+t\right) \right\vert d\mu \left( t\right) \right) 
\\ \notag
&+&\frac{2^M}{2^s}\left(2^M\sum_{A=1}^{r_s^1}\int_{I_{M}}2^{l^s_{j}}
\sum_{k=l^s_j}^{\infty}D_{2^{k}}\left( x+t\right) d\mu \left( t\right) \right).
\end{eqnarray}

If we denote by
\begin{eqnarray*}
	II_{\alpha^s _{A}}^{1}\left( x\right) &:=&2^{M}\int_{I_{M}}2^{\alpha^s
		_{A}}\left\vert K_{2^{\alpha^s _{A}}}\left( x+t\right) \right\vert d\mu \left(
	t\right) ,\text{\ } \alpha=l, \ \text{or} \ \alpha=t\\
	II_{l^s_{A}}^{2}\left( x\right)
	&:=&2^{M}\int_{I_{M}}2^{l^s_{A}}\sum_{k=l^s_{A}}^{\infty}D_{2^{k}}\left(
	x+t\right) d\mu \left( t\right),
\end{eqnarray*}
from \eqref{01} we can conclude that
\begin{eqnarray*}
	\left\vert \sigma _{n_{s_k}}a\left( x\right) \right\vert &\leq& 
	\frac{2^M}{2^s}\left(\sum_{A=1}^{r_s^1} II_{^{l^s_{A}}}^{1}\left(
	x\right) 
	+\sum_{A=1}^{r_s^2}II_{^{t^s_{A}}}^{1}\left( x\right) 
	+\sum_{A=1}^{r_s^1}II_{l_{A}^s}^{2}\left( x\right)
	\right)
\end{eqnarray*}
and
\begin{eqnarray*}
&&\sup_{2^s\leq n_{s_k}<2^{s+1}}\left\vert \sigma _{n_{s_k}}a\left( x\right) \right\vert \leq\frac{2^M}{2^s}\left(\sum_{A=1}^{r_s^1} II_{^{l^s_{A}}}^{1}\left(
	x\right) 
	+\sum_{A=1}^{r_s^2}II_{^{t^s_{A}}}^{1}\left( x\right) 
	+\sum_{A=1}^{r_s^1}II_{l_{A}^s}^{2}\left( x\right)
	\right).
\end{eqnarray*}
Hence,
\begin{eqnarray}\label{00}
	&&\int_{\overline{I_{M}}}\left(	\sup_{2^s\leq n_{s_k}<2^{s+1}}\left\vert \sigma _{n_{s_k}}a\left( x\right) \right\vert \right)^{1/2}d\mu(x)\\ \notag
	&\leq &\frac{2^{M/2}}{2^{s/2}}\left( \sum_{A=1}^{r^1_s}\int_{\overline{I_{M}}
	}\left\vert II_{l^s_{A}}^{1}\left( x\right) \right\vert ^{1/2}d\mu \left(
	x\right) +
	 \sum_{A=1}^{r^2_s}\int_{\overline{I_{M}}}\left\vert II_{t^s_{A}}^{1}\left( x\right)
	\right\vert ^{1/2}d\mu \left( x\right) \right.\\ \notag
	&&\left.+ \sum_{A=1}^{r^1_s}\int_{\overline{I_{M}}}\left\vert
	II_{l^s_{A}}^{2}\left( x\right) \right\vert ^{1/2}d\mu \left( x\right) \right).
\end{eqnarray}

Since $\sup_{s\in \mathbb{N}}r_s^1<r<\infty, \ \ \ \sup_{s\in \mathbb{N}}r_s^2<r<\infty,$
we obtain that \eqref{star4} holds so that Theorem \ref{theorem1} a) is proved if we can prove that 
\begin{equation}\label{11.0}
\int_{\overline{I_{M}}}\left\vert II_{l^s_{A}}^{2}\left( x\right) \right\vert
^{1/2}d\mu \left( x\right) \leq c<\infty , \ \ \ A=1,...,r_s^1 
\end{equation}
and
\begin{equation}\label{11.1}
\int_{\overline{I_{M}}}\left\vert II_{\alpha^s _{A}}^{1}\left( x\right)
\right\vert ^{1/2}d\mu \left( x\right) \leq c<\infty ,
\end{equation}
for all $\alpha^s _{A}=l^s_{A}, \ \ \ A=1,...,r_s^1$ and $\alpha^s _{A}=t^s_{A},  \ \ \ A=1,...,r_s^2$. Indeed, if \eqref{11.0} and \eqref{11.1} hold, from \eqref{00}   we get that
\begin{eqnarray*}
&&\int_{\overline{I_{M}}}\left(	\sup_{ n_{s_k}\geq2^{M}}\left\vert \sigma _{n_{s_k}}a\left( x\right) \right\vert \right)^{1/2}d\mu(x)\\
	&\leq&\sum_{s=M}^{\infty}\int_{\overline{I_{M}}}\left(\sup_{2^s\leq n_{s_k}<2^{s+1}}\left\vert \sigma _{n_{s_k}}a\left( x\right) \right\vert\right)^{1/2} \leq \sum_{s=M}^{\infty}\frac{c2^{M/2}}{2^{s/2}}<C<\infty.
\end{eqnarray*}

It remains to prove \eqref{11.0} and \eqref{11.1}.
Let $$t\in I_{M} \ \ \ \text{ and } \ \ x\in I_{l+1}\left( e_{k}+e_{l}\right).$$ 
If
$ 0\leq
k<l<\alpha^s_{A}\leq M \ \ \ \text{or } \  \ \ 0\leq k<l< M< \alpha^s_{A},$ then $x+t\in
I_{l+1}\left( e_{k}+e_{l}\right) $ and if we apply Lemma \ref{lemma2} we obtain
that 
\begin{equation}
K_{2^{\alpha^s _{A}}}\left( x+t\right) =0\text{ \ and \ \ }II_{\alpha^s
_{A}}^{1}\left( x\right) =0.  \label{10a}
\end{equation}
Let $0\leq k<\alpha^s _{A}\leq l< M.$
Then $x+t\in I_{l+1}\left( e_{k}+e_{l}\right) $  and if
we use Lemma \ref{lemma2} we get that 
$$2^{\alpha^s _{A}}\left\vert K_{2^{\alpha^s _{A}}}\left( x+t\right) \right\vert
\leq 2^{\alpha^s _{A}+k}$$
so that
\begin{equation}
II_{\alpha^s _{A}}^{1}\left(
x\right) \leq 2^{\alpha^s _{A}+k}.  \label{10b}
\end{equation}
Analogously to (\ref{10b}) we can prove that if $0\leq \alpha^s _{A}\leq
k<l< M$, then
\begin{equation*}
2^{\alpha^s _{A}}\left\vert K_{2^{\alpha^s _{A}}}\left( x+t\right) \right\vert
\leq 2^{2\alpha^s _{A}},\text{\ }t\in I_{M},\text{\ }x\in I_{l+1}\left(
e_{k}+e_{l}\right) 
\end{equation*}
so that
\begin{equation}
II_{\alpha^s _{A}}^{1}\left( x\right) \leq
2^{2\alpha^s _{A}},\text{\ }t\in I_{M},\text{\ }x\in I_{l+1}\left(
e_{k}+e_{l}\right) ,  \label{10c}
\end{equation}

Let $$t\in I_{M} \ \ \ \text{ and }  \ \ \  x\in I_M(e_k).$$ 
Let
$0\leq k<\alpha^s _{A}\leq M$ or $0\leq k<M\leq \alpha^s _{A}.$
Since
$x+t \in x\in I_M(e_k) $ and if we apply Lemma \ref{lemma2} we obtain
that 
$$2^{\alpha^s _{A}}\left\vert K_{2^{\alpha^s _{A}}}\left( x+t\right) \right\vert
\leq 2^{\alpha^s _{A}+k}$$
and
\begin{equation}
II_{\alpha^s _{A}}^{1}\left(
x\right) \leq 2^{\alpha^s _{A}+k}.  \label{10bb}
\end{equation}
Let
$0\leq \alpha^s _{A}\leq k<M.$
Since
$x+t \in x\in I_M(e_k) $ and if we apply Lemma \ref{lemma2}, then we find
that 
$$2^{\alpha^s _{A}}\left\vert K_{2^{\alpha^s _{A}}}\left( x+t\right) \right\vert
\leq 2^{2\alpha^s _{A}}$$
and
\begin{equation}
II_{\alpha^s _{A}}^{1}\left(
x\right) \leq 2^{2\alpha^s _{A}}.  \label{10bbb}
\end{equation}

Let $0\leq \alpha^s _{A}< M.$ By combining (\ref{1}) with (\ref{10a})-(\ref{10bbb}) for any $A=1,...,s$ we have that 
\begin{eqnarray*}
&&\int_{\overline{I_{M}}}\left\vert II_{\alpha^s _{A}}^{1}\left( x\right)
\right\vert ^{1/2}d\mu \left( x\right) \\
&=&\overset{M-2}{\underset{k=0}{\sum }}\overset{M-1}{\underset{l=k+1}{\sum }}
\int_{I_{l+1}\left( e_{k}+e_{l}\right) }\left\vert II_{\alpha^s
_{A}}^{1}\left( x\right) \right\vert ^{1/2}d\mu \left( x\right) +\overset{M-1
}{\underset{k=0}{\sum }}\int_{I_{M}\left( e_{k}\right) }\left\vert
II_{\alpha^s _{A}}^{1}\left( x\right) \right\vert ^{1/2}d\mu \left( x\right) \\
&\leq&c\overset{\alpha^s _{A}-1}{\underset{k=0}{\sum }}\overset{M-1}{\underset{%
l=\alpha^s _{A}}{\sum }}\int_{I_{l+1}\left( e_{k}+e_{l}\right) }2^{\left(
\alpha^s _{A}+k\right) /2}d\mu \left( x\right) +c\overset{M-2}{\underset{%
k=\alpha^s _{A}}{\sum }}\overset{M-1}{\underset{l=k+1}{\sum }}%
\int_{I_{l+1}\left( e_{k}+e_{l}\right) }2^{\alpha^s _{A}}d\mu \left( x\right) \\
&+&c\overset{\alpha^s _{A}-1}{\underset{k=0}{\sum }}\int_{I_{M}\left(
e_{k}\right) }2^{\left( \alpha^s _{A}+k\right) /2}d\mu \left( x\right) +c%
\overset{M-1}{\underset{k=\alpha^s _{A}}{\sum }}\int_{I_{M}\left( e_{k}\right)
}2^{\alpha^s _{A}}d\mu \left( x\right) \\
&\leq& c\overset{\alpha^s _{A}-1}{\underset{k=0}{\sum }}\overset{M-1}{\underset{%
l=\alpha^s _{A}+1}{\sum }}\frac{2^{\left( \alpha^s _{A}+k\right) /2}}{2^{l}}+c%
\overset{M-2}{\underset{k=\alpha^s _{A}}{\sum }}\overset{M-1}{\underset{l=k+1}{%
\sum }}\frac{2^{\alpha^s _{A}}}{2^{l}}\\
&+&c\overset{\alpha^s _{A}-1}{\underset{k=0}{
\sum }}\frac{2^{\left( \alpha^s _{A}+k\right) /2}}{2^{M}}+c\overset{M-1}{%
\underset{k=\alpha^s _{A}}{\sum }}\frac{2^{\alpha^s _{A}}}{2^{M}}
\leq C<\infty.
\end{eqnarray*}
Analogously we can prove that (\ref{11.1}) holds also 
for the case $\alpha^s _{A}\geq M.$  Hence, \eqref{11.1} holds and it remains to prove \eqref{11.0}.

Let 
$t\in I_{M} \  \text{ and } \  x\in
I_{i}\backslash I_{i+1}.$ \
If $i\leq l^s_{A}-1,$ since $x+t\in I_{i}\backslash
I_{i+1},$ by using (\ref{1dn}) we have that 
\begin{equation}\label{13a}
II_{l^s_{A}}^{2}\left( x\right) =0.  
\end{equation}

If $l^s_{A}\leq i< M,$ then by using (\ref{1dn}) we obtain that 
\begin{equation}\label{13b}
II_{l^s_{A}}^{2}\left( x\right) \leq
2^{M}\int_{I_{M}}2^{l^s_{A}}\sum_{k=l^s_{A}}^{i}D_{2^{k}}\left( x+t\right) d\mu
\left( t\right) \leq c2^{l^s_{A}+i}.  
\end{equation}

Let $0\leq l^s_{A}< M.$ By combining (\ref{1}), (\ref{13a}) and (\ref
{13b}) we get that 
\begin{eqnarray}\label{star5}
&&\int_{\overline{I_{M}}}\left\vert II_{l^s_{A}}^{2}\left( x\right) \right\vert
^{1/2}d\mu \left( x\right)
 =\left(
\sum_{i=0}^{l^s_{A}-1}+\sum_{i=l^s_{A}+1}^{M-1}\right)
\int_{I_{i}\backslash I_{i+1}}\left\vert II_{l^s_{A}}^{2}\left( x\right)
\right\vert ^{1/2}d\mu \left( x\right)\\ \notag
&\leq& c\sum_{i=l^s_{A}}^{M-1}\int_{I_{i}\backslash I_{i+1}}2^{\left(
l^s_{A}+i\right) /2}d\mu \left( x\right)\leq c\sum_{i=l^s_{A}}^{M-1}2^{\left( l^s_{A}+i\right) /2}\frac{1}{2^{i}}\leq
C<\infty .
\end{eqnarray}

If  $M\leq l^s_{A},$ then $i<M\leq l^s_{A}$ and apply \eqref{13a} we get that
\begin{eqnarray}\label{star6}
\int_{\overline{I_{M}}}\left\vert II_{l^s_{A}}^{2}\left( x\right) \right\vert
	^{1/2}d\mu \left( x\right)=0,
\end{eqnarray}
and also \eqref{11.0} is proved by just combining \eqref{star5} and \eqref{star6} so part a) is complete and we turn to the proof of b).

Under condition (\ref{cond2}), there exists an increasing sequence $\left\{ \alpha _{k}:\text{ }k\geq
0\right\} \subset \left\{ n_{k}\text{ }:k\geq 0\right\} $ of positive
integers, such that 
\begin{equation}
\sum_{k=1}^{\infty }1 /\vert A_{\left\vert\alpha _{k}\right\vert}^2  \leq c<\infty .  \label{2aaa}
\end{equation}

Let 
\begin{equation*}
F_{A}:=\sum_{\left\{ k;\text{ }\left\vert \alpha _{k}\right\vert <A\right\}
}\lambda _{k}a_{k},
\end{equation*}
where
\begin{equation*}
\lambda _{k}:=\frac{1}{\left\vert A_{\vert\alpha _{k}\vert}\right\vert}  \  \text{
and} \ 
a_{k}:=2^{\left\vert \alpha _{k}\right\vert }\left(
D_{2^{\left\vert \alpha _{k}\right\vert +1}}-D_{2^{\left\vert \alpha
		_{k}\right\vert }}\right) .
\end{equation*}

Since \ \ 
$
\text{supp}(a_{k})=I_{\left\vert \alpha _{k}\right\vert },
\ \ \ 
\left\Vert
a_{k}\right\Vert _{\infty }\leq 2^{2\left\vert \alpha _{k}\right\vert }=\mu (%
\text{supp }a_{k})^{-2} \ \
$ and
\begin{equation*}
S_{2^{A}}a_{k}=\left\{ 
\begin{array}{ll}
a_{k} & \left\vert \alpha _{k}\right\vert <A, \\ 
0\, & \left\vert \alpha _{k}\right\vert \geq A,%
\end{array}%
\right.  
\end{equation*}
if we apply Lemma \ref{lemma0} and (\ref{2aaa}) we can conclude that $F=\left(
F_{1},F_{2},...\right) \in H_{1/2}.$

It is easy to prove that
\begin{equation}\label{5aa}
\widehat{F}(j)  
=\left\{ 
\begin{array}{ll}
2^{\left\vert \alpha _{k}\right\vert }
/\left\vert A_{\left\vert\alpha _{k}\right\vert}\right\vert , & \text{\thinspace \thinspace }j\in \left\{
2^{\left\vert \alpha _{k}\right\vert },...,2^{_{\left\vert \alpha
_{k}\right\vert +1}}-1\right\} ,\text{ }k=0,1,... \\ 
0\,, & \text{\thinspace }j\notin \bigcup\limits_{k=0}^{\infty }\left\{
2^{_{\left\vert \alpha _{k}\right\vert }},...,2^{_{\left\vert \alpha
_{k}\right\vert +1}}-1\right\} .\text{ }%
\end{array}%
\right.
\end{equation}

Let $2^{\left\vert \alpha _{k}\right\vert }<j<\alpha _{k}.$ By using (\ref%
{5aa}) we get that%
\begin{eqnarray}\label{sn} \ \ \ \ \
S_{j}F=S_{2^{\left\vert \alpha _{k}\right\vert }}F+\sum_{v=2^{^{\left\vert
\alpha _{k}\right\vert }}}^{j-1}\widehat{F}(v)w_{v} =S_{2^{\left\vert \alpha
_{k}\right\vert }}F+\frac{\left( D_{j}-D_{2^{\left\vert \alpha
_{k}\right\vert }}\right) 2^{\left\vert \alpha _{k}\right\vert }}{\left\vert A_{{\left\vert\alpha _{k}\right\vert}}\right\vert}  
\end{eqnarray}
Let \
$\begin{matrix}
2^{\left\vert\alpha _{k}\right\vert}\le {{\alpha}_{{{s}_{n}}}}\le {2^{{\left\vert\alpha _{k}\right\vert}+1}}.
\end{matrix}$ \
Then, by using \eqref{sn} we find that
\begin{eqnarray}\label{7aaa} 
&&\sigma_{\alpha_{s_n}}F
=\frac{1}{\alpha_{s_n}}\sum_{j=1}^{2^{\left\vert \alpha_{k}\right\vert }}S_{j}F+\frac{1}{\alpha _{s_n}}
\sum_{j=2^{\left\vert \alpha_{k}\right\vert }+1}^{\alpha _{s_n}}S_{j}F \\ \notag
&=&\frac{\sigma_{2^{\left\vert \alpha _{k}\right\vert }}F}{\alpha _{s_n}}+\frac{\left( \alpha _{k}-2^{\left\vert
\alpha _{k}\right\vert }\right) S_{2^{\left\vert \alpha _{k}\right\vert }}F}{ \alpha _{s_n}}+\frac{2^{\left\vert \alpha
_{k}\right\vert } }{\left\vert A_{{\left\vert\alpha _{k}\right\vert}}\right\vert \alpha _{s_n}}\sum_{j=2^{_{\left\vert
\alpha _{k}\right\vert }}+1}^{\alpha_{s_n}}\left( D_j-D_{2^{\left\vert
\alpha _{k}\right\vert }}\right)\\ \notag
&:=&III_{1}+III_{2}+III_{3}.
\end{eqnarray}

Since
$$D_{j+2^{m}}=D_{2^{m}}+w_{_{2^{m}}}D_{j},\text{ \  when \ }j<2^{m}  $$
we obtain that

\begin{eqnarray}\label{9aaa} \ \ \ \ \ 
&&\left\vert III_{3}\right\vert =\frac{2^{\left\vert \alpha _{k}\right\vert
}}{\left\vert A_{\left\vert\alpha _{k}\right\vert}\right\vert \alpha _{s_n}}\left\vert \sum_{j=1}^{\alpha
_{s_n}-2^{_{\left\vert \alpha_{k}\right\vert }}}\left( D_{j+2^{_{\left\vert
\alpha _{k}\right\vert }}}-D_{2^{\left\vert \alpha _{k}\right\vert }}\right)
\right\vert \\ \notag
&=&\frac{2^{\left\vert \alpha _{k}\right\vert }}{\left\vert A_{\left\vert\alpha _{k}\right\vert}\right\vert \alpha _{s_n}}\left\vert \sum_{j=1}^{\alpha _{s_n}-2^{\left\vert \alpha_{k}\right\vert }}D_j\right\vert 
=\frac{2^{\left\vert \alpha_{k}\right\vert } }{\left\vert A_{\left\vert\alpha _{k}\right\vert}\right\vert \alpha_{s_{n}}}\left( \alpha
_{s_n}-2^{\left\vert \alpha_{k}\right\vert }\right) \left\vert K_{\alpha
_{s_n}-2^{\left\vert \alpha_{k}\right\vert }}\right\vert
\\ \notag
&\geq&\frac{1}{2\left\vert A_{\left\vert\alpha _{k}\right\vert}\right\vert }\left( \alpha_{s_n}-2^{\left\vert \alpha _{k}\right\vert }\right) \left\vert K_{\alpha_{s_{n}}-2^{\left\vert \alpha _{k}\right\vert }}\right\vert.
\end{eqnarray}
By combining the well-known estimates (see \cite{PTWbook})
	\begin{equation*}
\left\Vert S_{2^k}F\right\Vert_{H_{1/2}}\leq c_1 \left\Vert F\right\Vert_{H_{1/2}} \ \ \ \text{and } \ \ \ \left\Vert \sigma_{2^k}F\right\Vert_{H_{1/2}}\leq c_2 \left\Vert F\right\Vert_{H_{1/2}}, \ \ \ \ k\in \mathbb{N},
\end{equation*}
we obtain that
$$\left\Vert III_{1}\right\Vert
_{1/2}\leq C \ \ \ \text{and} \ \ \ \left\Vert III_{2}\right\Vert _{1/2}\leq C.$$ 

Let 
$\begin{matrix}
2^{\left\vert\alpha _{k}\right\vert}\le {{\alpha}_{{{s}_{1}}}}\le {{\alpha}_{{{s}_{2}}}}\le ...\le {{\alpha}_{{{s}_{r}}}}\le {2^{{\left\vert\alpha _{k}\right\vert}+1}}
\end{matrix}$ 
be natural numbers which generates the set 
$$
{{A}_{\left\vert\alpha _{k}\right\vert}}={\left\{ l_{1}^{\left\vert\alpha _{k}\right\vert},l_{2}^{\left\vert\alpha _{k}\right\vert},...,l_{r^1_{\left\vert\alpha _{k}\right\vert}}^{\left\vert\alpha _{k}\right\vert} \right\}}
\bigcup{\left\{ t_{1}^{\left\vert\alpha _{k}\right\vert},t_{2}^{\left\vert\alpha _{k}\right\vert},...,t_{{{r}_{\left\vert\alpha _{k}\right\vert}^2}}^{\left\vert\alpha _{k}\right\vert} \right\}}
$$
and choose number $\alpha_{s_{n}}=\sum_{i=1}^{r_{n}}\sum_{k=l_{i}^{n}}^{t_{i}^{n}}2^{k},$ where
\begin{eqnarray*}
t_{1}^{\left\vert\alpha _{k}\right\vert}\geq l_{1}^{\left\vert\alpha _{k}\right\vert}>l_{1}^{\left\vert\alpha _{k}\right\vert}-2\geq t_{2}^{\left\vert\alpha _{k}\right\vert}\geq
l_{2}^{\left\vert\alpha _{k}\right\vert}>l_{2}^{\left\vert\alpha _{k}\right\vert}-2\geq ...\geq t_{\left\vert\alpha _{k}\right\vert}^{\left\vert\alpha _{k}\right\vert}\geq l_{\left\vert\alpha _{k}\right\vert}^{\left\vert\alpha _{k}\right\vert}\geq 0,
\end{eqnarray*}
for some $1\leq n\leq r,$ such that 
$l^{\left\vert\alpha _{k}\right\vert}_u=l_i,  \  \text{for some} \ 1\leq u\leq r^1_{\left\vert\alpha _{k}\right\vert},  1\leq i\leq r_{\left\vert\alpha _{k}\right\vert}^1. $

Since
$\mu \left\{E_{l_{i}}\right\} \geq 1/2^{l_i+1},$
by using Corollary \ref{corollary1} we get that 
\begin{eqnarray}\label{low1}
\int_{E_{l_{i}}}\left\vert\widetilde{\sigma }^{\ast ,\nabla }F\right\vert^{1/2}d\mu \left( x\right)&\geq&	\int_{E_{l_{i}}}\left\vert \sigma _{\alpha _{s_n}}F(x)\right\vert^{1/2} d\mu \left( x\right)\\ \notag
&\geq& \frac{2^{\left(2l_{i}-6\right)/2}}{\sqrt{2}\left(\left\vert A_{ {\left\vert\alpha _{k}\right\vert}}\right\vert \right)^{1/2}}\frac{1}{2^{l_i+1}}\geq\frac{1}{2^5\left(\left\vert A_{{\left\vert\alpha _{k}\right\vert}}\right\vert\right)^{1/2}}.
\end{eqnarray}

On the other hand, we can also choose number $\alpha_{s_{n}},$ for some $1\leq n\leq r,$ such that 
$t^{\left\vert\alpha _{k}\right\vert}_u=t_i,  \  \text{for some} \ 1\leq u\leq r^2_{\left\vert\alpha _{k}\right\vert}, \ \  1\leq i\leq r_{\left\vert\alpha _{k}\right\vert}^2. $
According to the fact that 
$\mu \left\{
E_{t_{i}}\right\} \geq 1/2^{t_{i}+3},$ 
by using again Corollary \ref{corollary1} for some $\alpha _{k}$ and $1\leq i\leq r_s^2$ we also get that
\begin{eqnarray}\label{low2} \ \ \  \ \ \
	&&\int_{E_{t_{i}}}\left\vert\widetilde{\sigma }^{\ast ,\nabla }F\right\vert^{1/2} d\mu \left( x\right)\geq	\int_{E_{t_{i}}}\left\vert \sigma _{\alpha _{s_n}}F(x)\right\vert^{1/2} d\mu \left( x\right)\\ \notag
	&\geq& \frac{1}{\sqrt{2}\left(\left\vert A_{{\left\vert\alpha _{k}\right\vert}}\right\vert\right)^{1/2} }2^{\left(2t_{i}-6\right)/2}\frac{1}{2^{t_i+3}}\geq\frac{1}{2^7\left(\left\vert A_{{\left\vert\alpha _{k}\right\vert}}\right\vert\right)^{1/2}}.
\end{eqnarray}
By combining (\ref{7aaa})-(\ref{low2}) with Proposition \ref{lemma3} for sufficiently big $\alpha_k$ we get that 
\begin{eqnarray*}
&&\int_{G}\left\vert\widetilde{\sigma }^{\ast ,\nabla }F\right\vert^{1/2} d\mu \left( x\right)
\geq\left\Vert III_{3}\right\Vert
_{1/2}^{1/2}-\left\Vert III_{2}\right\Vert _{1/2}^{1/2}-\left\Vert
III_{1}\right\Vert _{1/2}^{1/2}\\
&\geq& \underset{i=1}{\overset{r^1_{\left\vert\alpha _{k}\right\vert}-1}{\sum }}\int_{E_{l_i}}\left\vert\widetilde{\sigma }^{\ast ,\nabla }F\right\vert^{1/2} d\mu \left( x\right)+\underset{i=1}{\overset{r^2_{\left\vert\alpha _{k}\right\vert}-1}{\sum }}\int_{E_{t_i}}\left\vert\widetilde{\sigma }^{\ast ,\nabla }F\right\vert^{1/2} d\mu \left( x\right)-2C\\
&\geq&\frac{1}{2^7\left(\left\vert A_{{\left\vert\alpha _{k}\right\vert}}\right\vert\right)^{1/2}}(r_{\left\vert\alpha _{k}\right\vert}^1+r_{\left\vert\alpha _{k}\right\vert}^2)-2C\geq \frac{\left(\left\vert A_{\left\vert\alpha _{k}\right\vert}\right\vert\right)^{1/2}}{2^8}\to \infty, \ \   \text{as} \ \  k\to \infty,
\end{eqnarray*}
so also part b) is proved and the proof is complete.
\end{proof}

\end{document}